\documentclass{birkmult}
 \newtheorem{thm}[equation]{Theorem}

 \theoremstyle{definition}
 \newtheorem{defn}[equation]{Definition}
 \theoremstyle{remark}
 \newtheorem{rem}[equation]{Remark}
 \newtheorem{siderem}[equation]{Side Remark}
 \newtheorem{ex}[equation]{Example}
 \numberwithin{equation}{section}
 
\usepackage{amsmath}
\usepackage{amssymb}
\usepackage{textcomp}
\usepackage{color}
\usepackage{scrpage2}
\usepackage{scrdate}
\usepackage{scrtime}
\usepackage{units, enumerate, array}
\usepackage[matrix,arrow,curve]{xy}

\begin{document}
\newcommand{\w}{\widetilde}
\newcommand{\R}{\mathbb{R}}
\newcommand{\Z}{\mathbb{Z}}
\renewcommand{\o}{\overline}
\renewcommand{\dim}{\mathrm{dim}\,}
\renewcommand{\l}{\left}
\renewcommand{\r}{\right}
%-------------------------------------------------------------------------
% editorial commands: to be inserted by the editorial office
%
%\firstpage{1}
%\volume{228}
%\DOI{003-0001}
%\seriesextraline{This is the Concrete Title of this Book\br H.E. R and S.T.C. W, Eds.}
%\firstpage{1}
\title[Reidemeister coincidence invariants of fiberwise maps]
 {\centering Reidemeister coincidence invariants\\ of fiberwise maps}
%----------Author 1
\author[Koschorke]{\centering Ulrich Koschorke}

\address{%
FB 6 (Mathematik)\\
Emmy-Noether-Campus\\
Universit\"{a}t\\
D 57068 SIEGEN\\
Germany}
\email{koschorke@mathematik.uni-siegen.de}

%\thanks{This work was completed with the support of our \TeX-pert.}
%----------Author 2

%\author{A Second Author}
%\address{The address of\br
%the second author\br
%sitting somewhere\br
%in the world}
%\email{dont@know.who.knows}


\date{10.07.2009}
\dedicatory{In memory of \\Jakob Nielsen (15.10.1890 - 3.8.1959).}
%%% ----------------------------------------------------------------------

\begin{abstract}
Given two fiberwise maps \(\ f_1,\ f_2\ \) between smooth fiber bundles over a base manifold \(B\), we develop techniques for calculating their Nielsen coincidence number. In certain settings we can describe the Reidemeister set of \(\ (f_1,f_2)\ \) as the orbit set of a group operation of \(\ \pi_1(B)\). The size and number of orbits captures crucial extra information. E.g. for torus bundles of arbitrary dimensions over the circle this determines the minimum coincidence numbers of the pair \(\ (f_1,f_2)\ \) completely. In particular we can decide when \(f_1\) and \(f_2\) can be deformed away from one another or when a fiberwise selfmap can be made fixed point free by a suitable homotopy. In two concrete examples we calculate the minimum and Nielsen numbers for all pairs of fiberwise maps explicitly. Odd order orbits turn out to play a special role.
\end{abstract}

%%% ----------------------------------------------------------------------
\maketitle
%%% ----------------------------------------------------------------------
%\tableofcontents
\section{Introduction and discussion of results}\label{sec1}
Let \(\ f: M\to M\ \) be a self-map of a manifold. The principal object of study in topological fixed point theory is the minimum number of fixed points among all maps homotopic to \(f\) (cf. [B], p. 19).

Similar minimum numbers also play a central role in coincidence theory where one considers two maps \(\ f_1,f_2: M\to N\ \) and tries to make the coincidence subspace 
\begin{equation}\label{1.1}
C(f_1,f_2)\ =\ \left\{x\in M\ \left| \ f_1(x)=f_2(x)\right.  \right\}
\end{equation}
of \(M\) as small as possible by deforming \(\ f_1\ \) and \(\ f_2\).

There is a good lower bound for such minimum numbers: the Nielsen number \(\ \mathrm{N}(f_1,f_2)\ \) (cf. e.g. [Ko1], 1.9). It depends only on the homotopy classes of \(f_1\) and \(f_2\) and counts the "essential" elements of a certain "Reidemeister set" (a geometric description of these concepts will be given below).

Two question arise. First:

\emph{Does the Nielsen number determine the minimum numbers?}

In fixed point theory the fundamental work of B. Jiang has shown that this may fail to be the case (cf. [Ji1], [Ji2], [Ji3] and also [Z], [Ke]); for further counterexamples in general coincidence theory see e.g. [Ko1], 1.17. Nevertheless, ever since the work of Wecken [W] many settings are known where minimum numbers are equal to (or can at least be described with the help of) \(\ \mathrm{N}(f_1,f_2)\ \) (for the stable dimension range see e.g. [Ko1], theorem 1.10).

Thus we are faced with the next question:

\emph{How can we compute the Nielsen number?}

In this paper we present an approach to this problem in the setting of fiberwise maps.

Consider the diagram
\begin{equation}\label{1.2}
\xymatrix{
& F_M \ar[d]_-{i_M} && F_N  \ar[d]\\
f_1,f_2:\!\!\!& M\ar[rr] \ar[rd]_-{p_M}&& N\ar[ld]^-{p_N}\\
&& B
}
\end{equation}
where \(\ p_M\ \) and \(\ p_N\ \) are the projection maps of smooth, locally trivial fibrations of closed connected manifolds and \(\ F_M\,, F_N\ \) denote the fibers over a basepoint \(\ \ast\ \) of \(B\). We assume that \(\ f_1,f_2\ \) are \emph{fiberwise maps}, i.e. that they make our diagram commute, and we are interested in their (fiberwise) minimum numbers
\begin{equation}\label{1.3}
\begin{split}
&\mathrm{MC}_B(f_1,f_2)\ =\ \min\left\{\#\ C(f_1',f_2')\ \left|\ f_i'\,\sim_B\, f_i,\ i=1,2\right.\right\}\\\text{and}\quad&\\
&\mathrm{MCC}_B(f_1,f_2)\ =\ \min\left\{\#\ \pi_0\left(C(f_1',f_2')\right)\ \left|\ f_i'\,\sim_B\ f_i,\ i=1,2\right.\right\}
\end{split}
\end{equation}
of \emph{coincidence points} and of \emph{coincidence components} (cf. \ref{1.1}). Here \(\ \sim_B\ \) denotes homotopies through fiberwise maps; if \(X\) ist any topological space, \(\ \pi_0(X)\ \) stands for its set of pathcomponents. Since we assume \(M\) to be compact \(\ \mathrm{MCC}_B\ \) is always finite and in general more likely to distinguish interesting coincidence phenomena than \(\ \mathrm{MC}_B\ \) (which is often infinite).

Next we describe the Reidemeister set and the Nielsen number of the pair \(\ (f_1,f_2)\ \) (for details compare [Ko1] and [GK]). Note that the coincidence set \(\ C=C(f_1,f_2)\ \) (c.f. \ref{1.1}) fits into the commuting triangle
\begin{equation}\label{1.4}
  \xymatrix{
	&**[r]\ar[d]^-{\mathrm{pr}}\save{%
	 \small{E_B(f_1,f_2):=\left\{(x,\theta)\in M\times P(N) \,\left|\,
	  \begin{matrix} p_{N}\, \circ\,\theta\,\equiv\, p_M(x);\\ \theta(0)=f_1(x),\\
	\,\theta(1)=f_2(x)\end{matrix}\right.\right\}}\restore} \\
	C \ar[ur]^-{\widetilde{g}} \ar[r]_-{g=\mathrm{incl}}&  M
  }
\end{equation}
Here \(\ P(N)\ \) denotes the space of all continuous paths \(\ \theta:\,[0,1]\to N\). Moreover for all \(\ x\in C\ \) we define
\[
\w{g}(x)=(x, \text{constant path at }f_1(x)=f_2(x)).
\]
Thus \(\ \w{g}\ \) lifts the inclusion map \(g\) in the Hurewicz fibration \(\ \mathrm{pr}\ \) (defined by the first projection).
Generically \(C\) is a closed smooth manifold, equipped with a description \(\ \o{g}\ \) of its stable normal bundle in terms of the pullback (under \(\ \w{g}\ \)) of the virtual vector bundle
\[
\w{\varphi}\ =\ \mathrm{pr}^\ast\left(f_1^\ast(TN)-TM-p_M^\ast(TB)\right)
\]
over \(\ E_B(f_1,f_2)\). These data yield a welldefined normal bordism class
\begin{equation}\label{1.5}
 \w{\omega}_B(f_1,f_2)\ =\ \left[C,\w{g},\o{g}\right]\ \in\ \Omega_\ast\left(E_B(f_1,f_2);\w{\varphi}\right)
\end{equation}
which vanishes if (and in the stable dimension range\[
\dim{M}<2(\dim{N}-\dim{B})-2
\]
only if) the pair \(\ (f_1,f_2)\ \) can be made coincidence free by fiberwise homotopies (cf. [GK], theorem 1.1; for further details concerning various versions of the \(\ \omega\ \)-invariant see e.g. [Ko1-5] and [KR]).

Note that the decomposition of the topological space \(\ E_B(f_1,f_2)\ \) into its pathcomponents \(Q\) induces a direct sum decomposition of normal bordism groups
\[
 \Omega_\ast\left(E_B(f_1,f_2);\w{\varphi}\right)=\bigoplus_{Q\in\pi_0\left(E_B(f_1,f_2)\right)} \Omega_\ast(Q; \w{\varphi}|Q)
\]
and a corresponding decomposition of the coincidence invariant \(\ \w{\omega}_B(f_1,f_2)\). We will count its nontrivial direct summands.

\begin{defn}(compare [GK], section 4)\label{1.6}\\
 The \textbf{geometric Reidemeister set} {\boldmath \(\ {\mathrm{R}}_{B}(f_1,f_2)\ \)} is the set \(\ \pi_0\left(E_B(f_1,f_2)\right)\ \) of all pathcomponents of \(\ E_B(f_1,f_2)\).

We call such an pathcomponent \(Q\) \textbf{essential} if it contributes nontrivially to \(\ \w{\omega}_B(f_1,f_2)\), i.e. if the triple \(\ \left(\w{g}^{-1}(Q),\w{g}|,\o{g}|\right)\ \) of (partial) coincidence data represents a nonvanishing bordism class in \(\ \Omega_\ast(Q; \w{\varphi}|Q)\).

The \textbf{Nielsen number} {\boldmath \({\mathrm{N}}_B(f_1,f_2)\)} is the number of essential elements \(\ Q\ \in\ \pi_0\left(E_B(f_1,f_2)\right)\ =\ \mathrm{R}_B(f_1,f_2)\).
\end{defn}

The Nielsen number depends only on the fiberwise homotopy classes of \(f_1\) and \(f_2\) (compare [Ko1], 1.9). Consequently
\begin{equation}\label{1.7}~\\
 \mathrm{N}_B(f_1,f_2)\ \leq\ \mathrm{MCC}_B(f_1,f_2)\ \leq\ \mathrm{MC}_B(f_1,f_2).
\end{equation}

\begin{ex}: \textit{Classical Nielsen fixed point theory in manifolds}\label{1.8}\\
  Here \(B\) consists only of one point and \((f_1,f_2)=(f,\mathrm{id})\) where \(f\) is a selfmap of a connected manifold \(M=N\). There is a bijection from the classical (''algebraic``) Reidemeister set \(\pi_1(M)\diagup\!\!\sim\,\,\ \)  onto \(R_B(f,\mathrm{id})=\pi_0\left(E_B(f,\mathrm{id})\right)\) and
\[\widetilde{\omega}_B(f,\mathrm{id})\, \in\, \, \Omega_0\left(E_B(f,\mathrm{id});\,\widetilde{\varphi}\right)=\bigoplus_{Q\in \mathrm{R_B}(f, \mathrm{id})}\mathbb{Z}\]
records the indices of all the Nielsen fixed point classes of \(f\) (cf. [Ko1]). In particular, our definition \ref{1.6} agrees with the classical definition of Nielsen and Reidemeister numbers (cf. [B]). Similary, \(\mathrm{MC}(f,\mathrm{id})\) is just the classical minimum number \(\mathrm{MF}(f)\) of fixed points (cf. [Br]).\(\hfill\Box\)
\end{ex}

Next we compose the fiber projections \(p_M\) (cf. \ref{1.2}) and pr (cf. \ref{1.4}) and obtain the Serre fibration
\begin{equation}\label{1.9}
 p_M\ \circ\ \mathrm{pr}\ :\ E_B(f_1,f_2)\ \longrightarrow\ B.
\end{equation}
Its fiber involves the maps 
\begin{equation}\label{1.10neu}
f_i|F_M\ :\ F_M\longrightarrow F_N,\ i=1,2. 
\end{equation}\\
\indent In view of definition \ref{1.6} the following questions come to mind very naturally.

\emph{Given a pathcomponent \(Q\) of \(\ E_B(f_1,f_2)\ \), can we decide - just on the basis of its topological properties (and without refering to \(\ \w{\omega}(f_1,f_2)\ \)) - whether it is essential? E.g. does the fibration \[ p_M\ \circ\ \mathrm{pr}\ |\ :\ Q\ \longrightarrow\ B\] give the necessary clues?}

More realistically: \emph{Can we determine \(\ \mathrm{N}_B(f_1,f_2)\ \) from knowing \(p_M\ \circ\mathrm{pr}\ \) on \emph{all} of \(\ E_B(f_1,f_2)\ \)?}

As we will see the answer to this last question is ''{yes!}`` in some very interesting cases. It involves the natural group action of the fundamental group \(\ \pi_1(B,\ast)\ \) on the set \(\ \pi_0(E|)\ \) (of pathcomponents of the fiber
\begin{equation}\label{1.11neu}
 E|\ :=\ \left(p_M\ \circ\ \mathrm{pr}\right)^{-1}\{\ast\}\ =\ E\left(f_1|F_M,\ f_2|F_M\right),
\end{equation}
compare \ref{1.10neu}) defined as follows. Given a loop \(\ \gamma:[0,1]\to B\ \) at \(\ \ast\ \) and a pathcomponent \(\ Q'\in\pi_0(E|)\ \), lift \(\gamma\) to a path \(\w{\gamma}\) in \(\ E_B(f_1,f_2)\ \) which starts in \(Q'\); then \(\ [\gamma]\cdot Q'\ \) is the pathcomponent of \(\ \w{\gamma}(1)\). Each orbit of this group action consists of the pathcomponents of \(\ E|\cap Q\ \) for some \(\ Q\ \in\ \pi_0\left(E_B(f_1,f_2)\right)\ \) and we can identify the orbit set with the geometric Reidemeister set \(\ \pi_0\left(E_B(f_1,f_2)\right)\).

We want to introduce an invariant which keeps track of this group operation (and not just of its orbit set).

\begin{defn}\label{1.10}
 Let \(\ G,G'\ \) be abelian groups.
\begin{enumerate}[(i)]
 \item A map \(\ h:G\to G'\ \) is called \textbf{affine} (or \textbf{affine isomorphism}, resp.) if it is the sum of a group homomorphism (or isomorphism, resp.) and a constant map.
\item Two group operations \(\ \beta,\ \beta'\ \) of \(\ \pi_1(B,\ast)\ \) on (the underlying sets of) \(G\) and \(G'\), resp., are \textbf{equivalent} if they are transformed into one another by some affine isomorphism \(\ h:G\to G'\ \) (i.e. for all \(\ y\in\pi_1(B,\ast)\ \) and \(g\in G\), we have \(\ h(\beta(y,g))=\beta'(y,h(g))\)).
\item Let \(\ \mathcal{R}_B\ \) denote the set of equivalence classes of the group operations of \(\ \pi_1(B)=\pi_1(B,\ast)\ \), acting on abelian groups via affine automorphisms.
\end{enumerate}
\end{defn}

Now assume that the fibers \(F_M, F_N\) are pathconnected and \(\ \pi_1(F_N)\ \) is abelian. Then there is a bijection between \(\ \pi_0(E|)\ \) and 
\begin{equation}\label{1.11}
 G\ :=\ \mathrm{Coker}\left(f_1|_\ast-f_2|_\ast\ :\ H_1(F_M;\Z)\to H_1(F_N;\Z)\right)
\end{equation}
which depends on the choice of a basepoint in \(E|\) but is unique up to an affine automorphism of \(G\). Thus the group action of \(\ \pi_1(B,\ast)\ \) on \(\ \pi_0(E|)\ \) gives rise to a welldefined element
\begin{equation}\label{1.12}
 \varrho_B(f_1,f_2)\ \in\ \mathcal{R}_B
\end{equation}
which we call the \textbf{Reidemeister invariant} of the pair \(\ (f_1,f_2)\ \) of fiberwise maps (for details see the proof of theorem \ref{2.5} below).

This can also be interpreted in terms of the algebraic Reidemeister equivalence relation. In section \ref{sec2} we will recall that there is a bijection between the geometric Reidemeister set \(\ \pi_0\l(E_B(f_1,f_2)\r)\ \) and its algebraic counterpart, the orbit set of a certain action of \(\ \pi_1(M) \). In view of the exact sequence
\begin{equation}\label{1.15neu}
\xymatrix@=1.3cm{
 \ldots\ \ar[r]&\ \pi_1(F_M)\ \ar[r]^-{i_{M\ast}}&\ \pi_1(M)\ \ar[r]&\ \pi_1(B)\ \ar[r]&\ 0
}
\end{equation}
we may break this action up into two steps. First restrict it to the ''vertical`` action of \(\ i_{M\ast}\l(\pi_1(F_M)\r) \). The orbit set of this partial action corresponds to the geometric Reidemeister set \(\ \pi_0(E|)\ \) of the pair \(\ \l(f_1|F_M,f_2|F_M\r)\) (cf. \ref{1.10neu}). The remaining ''horizontal`` action of \(\ \pi_1(B)\ \) defines \(\ \varrho_B(f_1,f_2)\).

We will test the strength of our Reidemeister invariant in the setting of certain torus bundles.

\begin{defn}\label{1.13}
 A \textbf{linear {\boldmath \(k\)}-torus bundle} (\(k=1,2,\ldots\)) is a smooth locally trivial fiber bundle with typical fiber \(\ T^k=\R^k\!\diagup \Z^k\ \) and structure group \(\ \mathrm{GL}(k,\Z)\ \) (which acts in the standard fashion on \(\ T^k\)).
\end{defn}
The following result was proved in [Ko5].
\begin{thm}\label{1.14}
 Given integers \(\ m,n,b\geq 1\ \), let \(M\) (and \(N\), resp.) be a linear \(k\)-torus bundle with \(k=m\) (and \(k=n\), resp.) over the \(b\)-dimensional sphere \(\ S^b\). 

Then for all fiberwise maps \(\ f_1,\ f_2\ :\ M\ \to\ N\ \) we have
\begin{align*}
  &\mathrm{MCC}_B(f_1,f_2)\ = \ \mathrm{N}_B(f_1,f_2)\qquad\qquad\qquad\text{and}\\
 &\mathrm{MC}_B(f_1,f_2)\ = \begin{cases}
\mathrm{N}_B(f_1,f_2) &\text{if }\ \mathrm{N}_B(f_1,f_2)=0\ \text{ or }\ m+b=n;\\
\infty &\text{else}.
\end{cases}
\end{align*}
Thus these minimum numbers are determined by the Nielsen number \(\ \mathrm{N}_B(f_1,f_2)\).

In turn, \(\ \mathrm{N}_B(f_1,f_2)\ \) is determined by the Reidemeister invariant \(\ \varrho_B(f_1,f_2)\ \) as follows:
\[
 \mathrm{N}_B(f_1,f_2)\ =\ \nu_B\l(\varrho_B(f_1,f_2)\r)
\]
where
\[
 \nu_B(G,\beta):=\begin{cases}
\l(\nu_\mathrm{odd}'+\nu_\mathrm{even}'+\nu_\infty'\r)(G,\beta) &\text{if }\ \mathrm{rank}\ G\leq \mathrm{rank}\l(\pi_1(B)\r);\\
\ 0 &\text{else};
\end{cases}
\]
and \(\ \nu_\mathrm{odd}',\nu_\mathrm{even}',\nu_\infty'\ \) count the orbits of \(\beta\) of odd, even or infinite cardinality, resp., mod \(\infty\) (i.e. replacing \(\infty\) by \(0\) and leaving finite numbers unchanged).
\end{thm}

This reduces the calculation of minimum numbers to distinguishing and counting orbits.

Next we describe the Reidemeister invariant in the case \(\ B=S^1\ \) which is of particular interest (for some details see section 4 in [Ko5]). Here \(N\) has the following form (up to an isomorphism of linear \(n\)-torus bundles):
\[
 N\ =\ \l(T^n\times [0,1]\r)\diagup(x,1)\ \sim\ \l(A_N(x),0\r)
\]
and \(\ p_N\ \) is the obvious projection to \(\ S^1\ =\ [0,1]\diagup 1\sim 0\ \); the gluing automorphism \(A_N\) is induced by the action of some matrix \(\ \o{A}_N\ \in\ \mathrm{GL}(n,\Z)\ \) on \(\ \l(\R^n,\Z^n\r)\ \). Similarly \(M\) has a gluing map given by \(\ \o{A}_M\ \in\ \mathrm{GL}(m,\Z)\ \).

Any linear map \(\ \o{L}\ :\ \Z^m\to\Z^n\ \) such that \(\ \o{L}\circ\o{A}_M\ =\ \o{A}_N\circ\o{L}\ \) induces a map \(\ L:T^m\to T^n\ \) which is compatible with the gluings. Moreover any integer vector \(\ \o{v}\in\Z^n\ \) determines a linear path \(\o{\gamma}_{\o{v}}\ :\ [0,1]\to\R^n\ \) from \(0\) to \(\o{v}\) and hence a loop \(\ \gamma_{\o{v}}\ \) in \(T^n\). Thus the map
\[
 T^m\times [0,1]\ \to\ T^n\times[0,1],\quad (x,t)\ \longrightarrow\ \l(Lx+\gamma_{\o{v}}(t),t\r)
\]
induces a fiberwise map
\begin{equation}\label{1.15}
  f_{L,\o{v}}\ :\ M\ \longrightarrow\ N.
\end{equation}
It is not hard to see that each fiberwise map \(f\) has this form after a suitable fiberwise homotopy (''straightening``). Actually \(L\) is determined by the restriction of \(f\) to a single fiber:
\[
 \o{L}\ =\ f|_\ast\ :\ \Z^m=H_1(F_M;\Z)\ \longrightarrow\ H_1(F_N;\Z)=\Z^n;
\]
in turn, \(\ \o{L}\ \) determines the homotopy class of \(\ f|\, :\, F_M\ \longrightarrow \ F_N.\)
The class \(\ [\o{v}]\in\Z^n\diagup\l(\o{A}_N-\mathrm{id}\r)\l(\Z^n\r)\ \) of the ''slope vector`` \(\ \o{v}\ \) characterizes \(\ f\circ(\text{zero section of }p_M)\ \) up to homotopy of sections in \(N\).

Given any pair \(\ (f_1,f_2)\ \) of fiberwise maps from \(M\) to \(N\), we may find homotopies \(\ f_i\ \sim_B\ f_{L_i,\o{v}_i}\ \), \(\ i=1,2\). Put
\begin{equation}\label{1.17}
 \o{L}\ := \ \o{L}_1-\o{L}_2 \quad \text{and}\quad \o{v}\ := \ \o{v}_1-\o{v}_2\in\Z^n.
\end{equation}
Then the affine automorphism \(\ \o{u}\ \longrightarrow\ \o{A}_N (\o{u}-\o{v})\ \) on \(\ \Z^n\ \) induces an affine automorphism (and hence a \(\ \l(\pi_1\l(S^1\r) =\Z\r)\ \)-action) \(\ \beta\ \) on \(\ G\ :=\ \Z^n\diagup\o{L}\l(\Z^m\r)\). We obtain the Reidemeister invariant
\begin{equation}\label{1.18}
 \varrho_B(f_1,f_2)\ =\ \l[(G,\beta)\r].
\end{equation}

\begin{rem}\label{1.19}
 Note that fiberwise maps \(\ f_1,f_2,f_3,\ldots\ \) into a linear \(n\)-torus bundle can be added and subtracted fiberwise. Clearly the pairs \(\ (f_1,f_2)\ \) and \(\ (f_1-f_3, f_2-f_3)\ \) have the same coincidence behavior. This is reflected in the formulae above. Thus we need to study only pairs consisting of \(\ f=f_{L,\o{v}}\ \sim_B\ f_1-f_2\ \) and \(\ f_{o,\o{o}}\ =\ (\text{zero section of }N)\circ p_M\ \sim_B\ f_2-f_2\).

If in particular \(\ M=N\ \) then the coincidences of \(\ (f_1,f_2)\ \) are just the fixed points of \(\ f_1-f_2+\mathrm{id}\).\(\hfill\Box\) 
\end{rem}

The problem of computing the order and the number of orbits can involve a wide variety of interesting algebraic phenomena. It turns out that odd order orbits play a special role.

\begin{ex}\label{1.20}
 Let \(\ M=N\ \) be the 2-torus bundle with gluing matrix \(\begin{pmatrix} 0& 1\\1& 0\end{pmatrix}\) which switches the coordinates in \(\ \R^2\). Given a pair \(\ (f_1,f_2)\ \) of fiberwise selfmaps we proceed as in \ref{1.17}. Then \(\ \o{L}=\begin{pmatrix}a& b\\ b &a\end{pmatrix}\ \) for some integers \(\ a,b\). Let \(\ c\ \) denote the sum of the two coordinates of the integer vector \(\ \o{v}\in\Z^2\). Then the triple \(\ (a,b,c)\in\Z^3\ \) characterizes the fiberwise homotopy class of \(\ f_1-f_2\ \sim\ f_{L,\o{v}}\).
\end{ex}
\begin{thm}\label{1.21}
 Define \(\ d=\mathrm{gcd}\,(a+b,c)\ \) to be the greatest common denominator of \(\ a+b\ \) and \(\ c\). Then 
\[
 \mathrm{MCC}_B(f_1,f_2)=\mathrm{N}_B(f_1,f_2)=\frac{d}{2}\cdot\begin{cases}
|a-b|& \text{if } a+b \text{ is an even multiple of }  d;\\
(|a-b|+1)& \text{if } a+b \text{ is an odd multiple of }  d.
\end{cases}
\]
\end{thm}

The proof is given in section \ref{sec4} below. It involves studying three distinct algebraic cases which arise from very different geometric situations. Thus the homogeneity of the final result comes somewhat as a surprise.

Next let us see what happens when we compose \(f_1\) and \(f_2\) with the inclusion map
\[
 \mathrm{in}\ :\ M\ \longrightarrow\ M\ \times_B\ N' 
\]
into the fiberwise product of \(N\) with the Klein bottle \(K\) or the 2-torus \(T\) (fibered in the standard fashion over \(\ B=S^1\)). It turns out that only Reidemeister classes which correspond to odd order orbits can possibly remain essential. All other Nielsen classes can be made empty by homotopies of \(\ \mathrm{in}\circ{f}_1\ \) and \(\ \mathrm{in}\circ{f}_2\ \) which exploit the extra space available in \(\ M\,\times_B\, N'\).

\begin{thm}\label{1.22}
 Given fiberwise maps \(\ f_i\ :\ M\ \longrightarrow\ M\ ,\ i=1,2,\ \) as in \ref{1.21}, consider 
\[
 \hat{f}_i\ =\ \mathrm{in}\circ f_i\ :\ M\ \longrightarrow\ M\ \times_B\ N'
\]
where \(\ N'=K\ \) or \(T\). Put \(\ d=\mathrm{gcd}\,(a+b,c)\).

If the extra factor \(N'\) is the Klein bottle and \(\ |a|\neq|b|\), then
\[
 \mathrm{MCC}_B\l(\hat{f}_1,\hat{f}_2\r)=\mathrm{MC}_B\l(\hat{f}_1,\hat{f}_2\r) = \begin{cases}
d& \text{if }\ (a+b)/d\ \text{ is odd};\\
0& \text{else}.
     \end{cases}
\]
In all other cases these minimum numbers vanish.
\end{thm}
The proofs of theorems \ref{1.21} and \ref{1.22} will be given in section \ref{sec4} below.
\section{The Reidemeister invariant {\boldmath\(\ \varrho_B(f_1,f_2)\)}}\label{sec2}
Given fiberwise maps \(\ f_1,f_2\,:\,M\ \longrightarrow\ N\ \) as in diagram \ref{1.2}, we discuss in this section several group actions whose orbit sets can be identified with the Reidemeister set \(\ \mathrm{R}_B(f_1,f_2)\ =\ \pi_0\l(E_B(f_1,f_2)\r)\ \) (cf. \ref{1.4} and \ref{1.6}).

Pick a point \(\ x_0\in F_M\ \) (cf. \ref{1.2}) and let \(\ \pi_{x_0} \ \) denote the set of homotopy classes (with endpoints kept fixed) of paths
\[
 \theta\ :\ \l([0,1],0,1\r)\ \longrightarrow\ \l(F_N,f_1(x_0),f_2(x_0)\r).
\]
Using homotopy lifting extension properties (compare [Wh], I.7.16) of the (Serre) fibration \(\ p_N\ \) we construct a welldefined action
\begin{equation}\label{2.1}
 \ast_B\ :\ \pi_1(M,x_0)\times \pi_{x_0}\ \longrightarrow \ \pi_{x_0}
\end{equation}
of the fundamtental group \(\ \pi_1(M,x_0)\ \) on the set \(\ \pi_{x_0}\ \) as follows (compare section 3 in [GK]). Given classes \(\ [\gamma]\ \in\ \pi_1(M,x_0)\ \) and \(\ [\theta]\ \in\ \pi_{x_0}\ \), lift the homotopy
\[
 h\ :\ [0,1]\times[0,1]\ \longrightarrow \ B,\quad h(s,t)\ =\ p_M\circ\gamma(s),
\]
to a map \(\ \w{h}\ :\ [0,1]\times[0,1]\ \longrightarrow\ N\ \) such that 
\[
 \w{h}(0,t)=\theta(t),\quad \w{h}(s,0)=f_1\circ\gamma(s),\quad \w{h}(s,1)=f_2\circ\gamma(s)
\]
for all \(\ s,t\in[0,1]\ \). Then the path \(\ \theta'\ \) defined by \(\ \theta'(t)\ := \ \w{h}(1,t)\ \) lies entirely in \(\ F_N\ \). We define
\[
 [\gamma]\ \ast_B\ [\theta]\ :=\ [\theta'].
\]

Clearly, \(\ \gamma\ \) and \(\ \w{h}\ \) together describe a path in \(\ \mathrm{E}_B(f_1,f_2)\ \) which joins \(\ (x_0,\theta)\ \) to \(\ (x_0,\theta')\ \) (since \(\ \l(\gamma(s),\ \w{h}(s,-)\r)\ \) is a point in \(\ \mathrm{E}_B(f_1,f_2)\ \) for every \(\ s\in[0,1]\ \)). Now note that \(\ \pi_{x_0}\ \approx\ \pi_0\l(\mathrm{pr}^{-1}(\{x_0\})\r)\ \) is just the set of pathcomponents of the fiber \(\ \mathrm{pr}^{-1}(\{x_0\})\ \subset\ \mathrm{E}_B(f_1,f_2)\ \) (cf. \ref{1.4}).
Thus two such pathcomponents belong to the same orbit of the \(\ \pi_1(M,x_0)\,-\,\)action if and only if they lie in the same pathcomponent of \(\ \mathrm{E}_B(f_1,f_2)\). We obtain (compare [GK], thm. 3.1)
\begin{thm}\label{2.2}
 For every pair \(\ f_1,f_2\ :\ M\ \longrightarrow\ N\ \) of fiberwise maps and for every choice \(\ x_0\in F_M\ \) of a basepoint there is a canonical bijection between the geometric Reidemeister set \(\ \pi_0\l(\mathrm{E}_B(f_1,f_2)\r) \ \) and the orbit set of the group action \(\ \ast_B\ \) of \( \ \pi_1(M,x_0)\ \) on the set \(\ \pi_{x_0}\ \).
\end{thm}

\begin{ex}\label{2.3}
 If \(\ B\ \) consists of a single point and \(\ f_1(x_0)=f_2(x_0)=:y_0\ \), then \(\ \pi_{x_0}\ =\ \pi_1(N,y_0)\ \) and \[
\quad\l[\gamma\r]\ \ast_B\ \l[\theta\r]\ =\ f_{1\ast}\l(\l[\gamma\r]\r)^{-1}\cdot\l[\theta\r]\cdot f_{2\ast}\l(\l[\gamma\r]\r)
\]
for all \(\ \l[\gamma\r]\ \in\ \pi_1(M,x_0)\ \) and \(\ [\theta]\in\pi_1(N,y_0)\ \); thus the orbit set of our group action agrees with the standard (''algebraic``) Reidemeister set. \(\hfill\Box\)
\end{ex}

Next we simplify the algebraic description of the Reidemeister set while preserving the orbit structure approach at least partially. Note that \(\ i_{M\ast}\l(\pi_1(F_M,x_0)\r)\ \) is a \textit{normal} subgroup of \(\ \pi_1(M,x_0)\ \) (cf. \ref{1.2} and \ref{1.15neu}). Therefore \(\ \ast_B\ \) induces a welldefined group action of \(\ \pi_1(B,\ast)\ \cong\ \pi_1(M,x_0)\diagup i_{M\ast}\l(\pi_1(F_M,x_0)\r)\ \) on the orbit set of the action of \(\ i_{M\ast}\l(\pi_1(F_M,x_0)\r)\ \) on \(\ \pi_{x_0}\ \) (obtained by restricting \(\ \ast_B\ \)). But this orbit set can be identified with \(\ \pi_0\l(E|\r)\ \) (cf. \ref{1.11neu}), in analogy to theorem \ref{2.2} (compare [Ko1], 2.1). We obtain the interpretation of the geometric Reidemeister set \(\ \pi_0\l(\mathrm{E}_B(f_1,f_2)\r)\ \) in terms of the \(\ \pi_1(B,\ast)\,-\,\)action on \(\ \pi_0\l(E|\r)\ \) as described in the introduction (cf. \ref{1.11neu}).

Now assume that the fibers \(\ F_M\ \) and \(\ F_N\ \) are connected and that \(\ \pi_1(F_N)\ \) is abelian. We will endow \(\ \pi_0\l(E|\r)\ \) with a group structure, unique up to affine isomorphisms. In particular, then the group rank is welldefined; this plays a role e.g. when we want to apply theorem \ref{1.14}.

Pick a basepoint \(\ (x_0,\theta_0)\ \) of \(\ E|\ \) and consider the map
\[
 \psi\ :\ \pi_1(F_N,f_2(x_0))\ \longrightarrow\ \pi_0\l(E|\r)
\]
where, by definition, \(\ \psi\l(\l[\eta\r]\r)\ \) is the pathcomponent of \(\ (x_0,\theta_0\ast \eta)\ \), determined by the concatenated path
\[
 \xymatrix@=1.5cm{
f_1(x_0)\ \ar@{|->}[r]^-{\theta_0}&f_2(x_0)\ar@{|->}[r]^-{\eta}&f_2(x_0)
},
\]
\([\eta]\in\pi_1\l(F_N,f_2(x_0)\r)\ \). The connectedness of \(\ F_M\ \) implies that \(\ \psi\ \) is onto. Furthermore two elements \(\ [\eta_1],\ [\eta_2]\in\pi_1\l(F_N,f_2(x_0)\r)\ \cong\ H_1\l(F_N;\Z\r)\ \) have the same value under \(\ \psi\ \) if and only if there exists a loop \(\ \gamma\ \) at \(\ x_0\ \) in \(\ F_M\ \) such that the loop
 \[
\theta_0\ \ast\ \eta_1\ \ast\ \l(f_2\circ\gamma\r)\ \ast\ \eta_2^{-1}\ \ast\ \theta_0^{-1}\ \ast\ \l(f_1\circ\gamma\r)^{-1}
\]
is nullhomotopic in \(\ F_N\ \) or, equivalently, the cycle \(\ \eta_1-\eta_2-f_1\circ\gamma+f_2\circ\gamma\) is nullhomologuous. Thus \(\ \psi\ \) induces a bijection \(\ \psi_{x_0,\theta_0}\ \) from the group
\begin{equation}\label{2.4}
 \boldsymbol{G}\ :=\ \mathrm{Coker}\l(f_1|_\ast - f_2|_\ast\ :\ H_1\l(F_M;\Z\r)\ \longrightarrow\ H_1\l(F_N;\Z\r)\r)
\end{equation}
onto the set \(\ \pi_0\l(E|\r)\ \).

Next let us check that \(\ \psi_{x_0,\theta_0}\ \) transforms the action of any given element \(\ [\sigma]\in\pi_1(B,\ast)\ \) on \(\ \pi_0\l(E|\r)\ \) into an affine automorphism of \(\ G\ \). Let \(\ A\ :\ F_N\ \longrightarrow\ F_N\ \) be a gluing map for the bundle over \(\ S^1\ =\ [0,1]\diagup 1\sim 0\ \) obtained by pulling \(\ (N,p_N)\ \) back via \(\ \sigma\ \); thus
\[
 \sigma^\ast(N)\ \cong\ \l(F_N\times[0,1]\r)\diagup(x,1)\sim(A(x),0).
\]
Also lift \(\ \sigma\ \) to a loop \(\ \w{\sigma}\ \) at \(\ x_0\ \) in \(\ M\ \). Given \(\ [\eta]\ \in\ \pi_1\l(F_N,\,f_2(x_0)\r)\ \cong\ H_1\l(F_N;\Z\r)\ \), the action of \(\ \w{\sigma}\ \) first takes \(\ (0,\,\theta_0\ast\eta)\ \) to \(\ \l(1, \,\tau_1^{-1}\ast\theta_0\ast\eta\ast\tau_2\r)\ \); here \(\ \tau_i\ :\ [0,1]\ \to\ F_N\ \) is a path whose graph in \(\ F_N\times[0,1]\ \) corresponds to the map \(\ f_i\circ\w{\sigma}\ \) into \(\ N\ \), \(i=1,2\). After applying the gluing map we see that
\[
 \l[\w{\sigma}\r]\ \ast_B\ \psi_{x_0,\theta_0}([\eta])\ =\ \psi_{x_0,\theta_0}([\o{\eta}])
\]
where \(\ \o{\eta}\ \) is the sum of the cycles \(\ A\circ\eta\ \) and  \(\ -\theta_0-A\circ\tau_i+A\circ\theta_0+A\circ\tau_2\ \). Thus the operation \(\ \ast_B\ \) gives rise, via \(\ \psi_{x_0,\theta_0}\ \), to a group action \(\ \boldsymbol{\beta}\ \) of \(\ \pi_1(B,\ast)\ \) on \(\ G\ \) by affine automorphisms.

Let us also examine how this depends on the choice of the basepoint \(\ (x_0,\theta_0)\ \) of \(\ E|\ \). Any path \(\ \gamma\ \) in \(\ F_M\ \) from \(\ x_0\ \) to some other point \(\ x_0'\ \) lifts to a path in \(\ E|\ \) from \(\ (x_0,\,\theta_0\ast\eta)\ \) to \(\ \l(x_0',\, \theta_0'\ast(f_2\circ\gamma)^{-1}\ast\eta\ast(f_2\circ\gamma)\r)\ \) for some element of \(\ \theta_0'\ \in\ \pi_{x_0'}\ =\ \mathrm{pr}^{-1}\{x_0'\}\ \) which does not depend on \(\ [\eta]\in G\ \). Thus \(\ \psi_{x_0,\theta_0}\ =\ \psi_{x_0',\theta_0'}\ \). Any other choice of a basepoint \(\ (x_0',\theta_0'')\ \) over \(\ x_0'\ \) gives rise to a translation in \(\ G\ \).

Finally we observe that the \(\ \pi_1(B,\ast)\ \)-action \(\ \beta\ \) on \(\ G\ =\ \mathrm{Coker}\l(f_1|_\ast-f_2|_\ast\r)\ \) is compatible with fiberwise homotopies \(\ h_i\ :\ M\times[0,1]\ \longrightarrow\ N\ \) from \(\ f_i\ \) to some map \(\ f_i',\ i=1,2\ \). Indeed, define a fiber homotopy equivalence
\[
 H\ :\ \mathrm{E}_B(f_1,f_2)\ \longrightarrow\ \mathrm{E}_B(f_1',f_2')
\]
by \(\ H(x,\theta)\ =\ \l(x,\,\l(h_1(x,-)\r)^{-1}\ast\theta\ast h_2(x,-)\r)\ \), \(\ (x,\theta)\ \in\ \mathrm{E}_B(f_1,f_2)\ \). It commutes with bijections as described in \ref{2.4}, i.e.
\[
 H|\,\circ\,\psi_{x_0,\theta_0}\ =\ \psi'_{x_0,\theta'_0}\ :\ G\ \stackrel{\approx}{\longrightarrow}\ E'|\,:=E\l(f_1'|F_M,\,f_2'|F_M\r)
\]
(compare \ref{1.10neu} and \ref{1.11neu}) where we choose \(\ (x_0,\theta'_0)\ :=\ H(x_0,\theta_0)\ \) to be the basepoint of \(\ E'|\ \). Moreover both bijections yield the same \(\ \pi_1(B,\ast)\ \)-action on \(\ G\ \). 

We have proved
\begin{thm}\label{2.5}
 Assume that the fibers \(\ F_M\ \) and \(\ F_N\ \) are connected and that\(\ F_N\ \) has an abelian fundamental group. Then the equivalence class (cf. \ref{1.10})
\[
 \varrho_B(f_1,f_2)\ :=\ \l[\l(G,\beta\r)\r]\ \in\ \mathcal{R}_B
\]
depends only on the fiberwise homotopy classes of \(\ f_1\ \) and \(\ f_2\ \).
\end{thm}
We call it the \textbf{Reidemeister invariant} of the pair \(\ (f_1,f_2)\ \). The orbits of \(\ (G,\beta)\ \) not only can be used to label the Nielsen coincidence classes but may even capture much further essential information.

\section{Counting orbits}\label{sec3}
In this section we study the size and the number of orbits of certain affine automorphisms which occur e.g. in the setting of example \ref{1.20}.

\begin{thm}\label{3.1}
 Let \(\ w\ \) be a fixed element of an abelian group \(G\) and let \(\ \alpha\ \) be a group automorphism of \(G\) such that \(\ \alpha^2\ \) is the identity map \(\mathrm{id}\). Define \(\ \beta\ :\ G\ \longrightarrow\ G\ \) by
\[
 \beta(u)\ =\ \alpha(u)+w,\quad u\in G.
\]
Let \(\ q_0\ \) denote the cardinality of the subgroup of \(G\) generated by \(\ (\alpha+\mathrm{id})(w)\).

If \(\ q_0\ \) is infinite then so is every orbit
\[
 \{\ldots,\beta^{-1}(u),u,\beta(u),\beta^2(u),\ldots\}
\]
of \(\beta\).

If \(\ q_0\ \) is finite, then \(\ q_0\ \) (if \(\ q_0\equiv 1(2)\ \)) and \(\ 2q_0\ \) are the only possible orders of any orbit of \(\beta\). Odd order orbits exist if and only if
\[
q_0\ \text{is odd and }\ q_0 w\ \in\ (\alpha-\mathrm{id})(G).\tag{*}
\]
(This holds e.g. if \(\ q_0\equiv 1(2)\ \) and \(\ (\alpha-\mathrm{id})(G)\ =\ \mathrm{ker}\,(\alpha+\mathrm{id})\).) The numbers of \emph{all} orbits of odd (or even) order are

\begin{align*}
 \nu_\mathrm{odd}(\beta)\ &=\ \begin{cases}
\l(\#\mathrm{ker}\, (\alpha-\mathrm{id})\r)\diagup q_0 & \text{if (*) holds};\\
\ 0 &\text{otherwise};
 \end{cases}\\
 \nu_\mathrm{even}(\beta)\ &=\ \begin{cases}
\l(\#\l(G\diagdown\mathrm{ker}\,(\alpha-\mathrm{id})\r)\r)\diagup 2q_0 & \text{if (*) holds};\\
\# G\diagup 2q_0 &\text{otherwise}.
 \end{cases}
\end{align*}
\end{thm}
\begin{proof}
 For all natural numbers \(q\) and all \(\ u\in G\ \) we have
 \begin{subequations}
\begin{align}
  \beta^{2q}(u)\ &=\ u+q(\alpha+\mathrm{id})(w);\label{3.2a}\\
\beta^{2q+1}(u)\ &=\ \alpha(u)+(q(\alpha+\mathrm{id})+\mathrm{id})(w).\label{3.2b}
\end{align}
 \end{subequations}

In view of the first identity orbits can be finite only if \(\ q_0\ \) is finite. Then \(\ \beta^{2q_0}=\mathrm{id}\). Therefore \(\ 2q_0\ \) is an upper bound for all orbit orders and, in fact, their only possible \emph{even} value. Moreover, each of the following statements implies the next one (compare \ref{3.2b}):
\begin{enumerate}[(i)]
 \item there exists an orbit of odd order \(\ 2q+1\);
\item \(\beta^{2q+1}\ \) has a fixed point;
\item \((q(\alpha+\mathrm{id})+\mathrm{id})(w)=(2\,q+1)w+(\alpha-\mathrm{id})(qw)\ \in\ (\alpha-\mathrm{id})(G)\);
\item \((q(\alpha+\mathrm{id})+\mathrm{id})(w) \ \in\ \mathrm{ker}\,(\alpha+\mathrm{id})\ \) or, equivalently, \(\ (2q+1)(\alpha+\mathrm{id})(w)=0;\)
\item \(2q+1\ \in\ q_0\cdot\Z\).
\end{enumerate}

Due to the upper bound \(\ 2q_0\ \), condition (i) can hold only in the case \(\ 2q+1=q_0\). Here it is equivalent to conditions (ii) and (iii). Moreover the union of all odd order orbits is the fixed point set 
\[
 (\alpha-\mathrm{id})^{-1}\l\{-(q(\alpha+\mathrm{id})+\mathrm{id})(w)\r\}
\]
of \(\ \beta^{q_0}\); if nonempty, it has the same cardinality as the fixed point set \(\ \mathrm{ker}\, (\alpha-\mathrm{id})\ \) of \(\alpha\). This implies our claim concerning the orbit numbers.

Clearly \(\ (\alpha-\mathrm{id})(G)\ \subset\ \mathrm{ker}\,(\alpha+\mathrm{id})\); if these groups coincide, then all five conditions (i), \ldots, (v) hold provided \(\ 2q+1=q_0\).
\end{proof}

\section{Torus bundles and orbit analysis}\label{sec4}
We are now prepared to carry out the explicit orbit counts needed to deduce theorems \ref{1.21} and \ref{1.22} from \ref{1.14}. For this we use the notations of sections \ref{sec1} and \ref{sec3} and describe elements of quotient groups by their representatives, put into square brackets, as usual.

\noindent\emph{Proof of theorem \ref{1.21}.}
The Reidemeister invariant \(\ \varrho_B(f_1,f_2)\ \) in question (cf. \ref{1.11} and \ref{1.12}) is represented by the group
\begin{equation}\label{4.1}
 G\ =\ \mathrm{Coker}\, \o{L}\ =\ \Z^2\diagup\l(\Z(a,b)+\Z(b,a)\r)
\end{equation}
together with the affine automorphism \(\ \beta=\alpha+w\ \) where \(\ \alpha([u_1,u_2])=[u_2,u_1]\ \) and \(\ w=-\alpha([\o{v}])\).

Although the final result in Theorem \ref{1.21} looks quite homogeneous we have to distinguish three cases (labelled by the codimension of \(\ \o{L}\l(\R^2\r)\ \) in \(\ \R^2\)) which correspond to very different geometric situations.\\

\noindent \textit{Case 0: }\(|a|\neq|b|\) (i.e. \(\ \o{L}\ \) is injective).\\
Since \(\ (a+b,a+b)=(a,b)+(b,a)\ \) in \(\ \Z^2\ \) we have the sequence of welldefined homomorphisms
\begin{equation}\label{4.2}
 \xymatrix@=1.5cm{
0\ \ar[r]\ &\ \Z_{|a+b|}\ \ar[r]^-{\triangle} \ &G\ \ar[r]^-{r}\ &\ \Z_{|a-b|}\ \ar[r]&\ 0
}
\end{equation}
where \(\ \triangle([z])=[z,z]\ \) and \(\ r([u_1,u_2])=[u_1-u_2]\). Note that each class in the quotient group \(\ G\ \) can be represented by a unique point (with integer coordinates) in the parallelogram
\[
 \l\{s(a,b)+t\,(a+b,a+b)\ |\ s,t\in[0,1)\r\}\ \subset\ \R^2.
\]
Thus the order of \(\ G\ \) equals \(\ |a-b|\cdot|a+b|=|a^2-b^2|\ \) and the sequence \ref{4.2} is exact.

\begin{siderem}\label{4.3neu}
 The group \(G\) is isomorphic to the direct sum of the cyclic groups of orders \(\ k=\mathrm{gcd}\,(a,b)\ \) and \(\ |a^2-b^2|/ k\). In particular, \(G\) is cyclic if and only if \(a\) and \(b\) are relatively prime.

The following table describes \(G\) up to isomorphisms for low values \(\ |a|>|b|>0\).\\
\begin{center}
% use packages: array
\begin{tabular}{c|c|c|c|c|c|c}
\((a,b)\) & \((\pm 2,\pm 1)\) & \((\pm 3,\pm 1)\) & \((\pm 3,\pm 2)\) & \((\pm 4,\pm 1)\) & \((\pm 4,\pm 2)\) & \((\pm 4,\pm 3)\) \\ \hline
\(G\) & \(\Z_3\) & \(\Z_8\) & \(\Z_5\) & \(\Z_{15}\) & \(\Z_2\oplus \Z_6\) & \(\Z_7\)
\end{tabular}
\end{center}~\\
Since we will not need all this we leave the proof to the reader. \(\hfill\Box\)
\end{siderem}

Next note that \(\ (\alpha+\mathrm{id})([u_1,u_2])\ =\ \triangle([u_1+u_2])\ \) (cf. \ref{4.1} and \ref{4.2}). Hence
\[
 (\alpha+\mathrm{id})(G)\ =\ \triangle(G)\ \cong\ \Z_{|a+b|}
\]
and \(\ \#\mathrm{ker}\,(\alpha+\mathrm{id})\ =\ |a-b|\). A similar argument holds for the switching involution \(\ \alpha'\ \) on \(\ G'\ :=\ \Z^2\diagup\l(\Z(a,-b)+\Z(-b,a)\r)\). But complex conjugation, when restricted to \(\ \Z^2\ \subset\ \R^2\ =\ \mathbb{C}\), induces an isomorphism \(\ G'\ \cong\ G\ \) which tranforms \(\ \alpha'+\mathrm{id}\ \) into \(\ -\alpha+\mathrm{id}\ \). Therefore
\[
 \#\mathrm{ker}\,(\alpha-\mathrm{id})\ =\ \#\mathrm{ker}\,(\alpha'+\mathrm{id})\ =\ |a+b|
\]
and the groups \(\ (\alpha-\mathrm{id})(G)\ \subset\ \mathrm{ker}\,(\alpha+\mathrm{id})\), both having the cardinality \(\ |a-b|\), must agree. Thus an important criterion in theorem \ref{3.1} is satisfied. Theorem \ref{1.21} (as far as Case 0 is concerned) now follows since
\[
 (\alpha+\mathrm{id})(w)\ =\ -[c,c]\ =\ -\triangle([c])
\]
has the same order in \(G\) as \(\ [c]\ \) has in \(\ \Z_{|a+b|}\ \), namely \(\ q_0\ =\ |a+b|/d\ \) (compare \ref{4.2}).

\noindent\emph{Geometric meaning.}
Whenever \(\ f_i'\ \sim_B\ f_i,\ i=1,2,\ \) the coincidence locus \(\ C(f'_1,f'_2)\ \) intersects each fiber \(\ F\ \) of \(\ p_M\ \) in a set of cardinality at least \(\ |a^2-b^2|>0\). Thus \(\ \mathrm{MC}_B(f_1,f_2)=\infty\). However \(\ \mathrm{MCC}_B(f_1,f_2)\ \) will often be much smaller than \(\ |a^2-b^2|\ \) since many points of \(\ C(f_1',f_2')\ \cap\ F\ \) may lie in the same pathcomponent of \(\ C(f'_1,f'_2)\).\\

\noindent\textit{Case 1: \(\ a= -\varepsilon b\neq0\ \) where \(\ \varepsilon=\pm 1\) (this means that \(\ \o{L}\l(\R^2\r)\ \) is \(1\)-dimensional)}.\\
We have the isomorphism
\begin{equation}\label{4.4neu}
 G\ =\ \l(\Z(-\varepsilon,1)\oplus\Z(1,0)\r)\diagup\Z(-\varepsilon\, a,a)\ \stackrel{\cong}{\longrightarrow}\ \Z_{|a|}\oplus\Z
\end{equation}
which maps \(\ [u_1,u_2]\ \) to \(\ ([u_2],u_1+\varepsilon\, u_2)\ \) and which transforms \(\ \alpha\ \) into the involution
\[
 ([s],t)\ \longrightarrow\ ([t]-\varepsilon\, [s],\varepsilon\, t),\quad ([s],t)\ \in\ \Z_{|a|}\oplus\Z.
\]
As in [Ko5] we must distinguish two subcases according to whether the gluing map \(\ \o{A}_M\ \) preserves the coorientations of \(\ \o{L}\l(\R^2\r)\ \) in \(\ \R^2\ \) or not.\\

\noindent\textit{Subcase 1+: \(\ \varepsilon=+1\).}\\
Here \(\ \mathrm{ker}\,(\alpha-\mathrm{id})\ \) is infinite and 
\[
 \#\l(G\diagdown \mathrm{ker}\,(\alpha-\mathrm{id})\r)=\begin{cases}
0&\text{if }\ |a|=1;\\
\infty &\text{else}.
\end{cases}
\]
Thus the \(\ \mathrm{mod}\,\infty\ \) orbit numbers \(\ \nu'_{\mathrm{odd}}\ \) and \(\ \nu'_{\mathrm{even}}\ \) of \(\ (G,\beta)\ \) vanish (cf. \ref{1.14} and \ref{3.1}). Infinite orbits exist only if \(\ c\neq 0\ \), and then each of them intersects the subset \(\ \Z_{|a|}\oplus\l\{1,\ldots,|c|\r\}\ \) of \(\ \Z_{|a|}\oplus\Z\ \) in a unique point. Thus 
\[
 \nu_B(G,\beta)=\nu_\infty'(G,\beta)=|a|\cdot|c|=\frac{d}{2}\cdot|a-b|
\]
where \(\ d=\mathrm{gcd}\,(0,c)=|c|\geq 0\).

\noindent\emph{Geometric meaning.}
If \(\ c\ \) vanishes then \(\ f_{L,\o{v}}(M)\ \subset\ M\ \) is a coorientable submanifold of codimension \(1\) and we can push \(\ f_{0,\o{0}}\ \) away from \(\ f_{L,\o{v}}(M)\ \) into some ''positive direction``. Thus
\[
 \mathrm{MCC}_B(f_1,f_2)\ =\ \mathrm{MCC}_B\l(f_{L,\o{v}},f_{0,\o{0}}\r)\ =\ 0
\]
in this case (compare remark \ref{1.19}).\\

\noindent\textit{Subcase 1-: \(\ \varepsilon= -1\).}\\
Here \(\ a=b\neq0\ \) and the isomorphism 
\[
 (\alpha+\mathrm{id})(G)\ =\ \mathrm{ker}\,(\alpha-\mathrm{id})\ \cong\ \Z_{|a|}\oplus\{0\}
\]
(cf. \ref{4.4neu}) sends \(\ (\alpha+\mathrm{id})(w)\ =\ -[c,c]\ \) to \(\ -\l([c],0\r)\). Thus \(\ q_0\ \) (cf. \ref{3.1}) is equal to \(\ |a|/\mathrm{gcd}\,(a,c)\). Define also \(\ c'\ :=\ c/\mathrm{gcd}\,(a,c)\). Note that \(\ (\alpha-\mathrm{id})(G)\ \) agrees with the kernel of the sum homomorphism from \(G\) to \(\ \Z_{2\,|a|}\ \) which maps \(\ [u_1,u_2]\ \) to \(\ [u_1+u_2]\ \) and hence \(w\) to \(\ -[c]\in\Z_{2\,|a|}\). Therefore condition (*) in theorem \ref{3.1} holds if and only if \(\ q_0\equiv 1(2)\ \) and \(\ q_0\cdot c = |a|\cdot c'\equiv 0\ \mathrm{mod}\,2\,|a|\ \) or, equivalently, \(c'\) is even. In turn, this implies that \(\ d=2\,\mathrm{gcd}\,(a,c)\ \) and that \(\ 2\,|a|=q_0\,d\ \) is an odd multiple of \(d\). On the other hand, if \(\ c'\not\equiv 0(2)\ \) then \(\ d=\mathrm{gcd}\,(a,c)\ \) and \(\ 2\,|a|=2\,q_0\,\mathrm{gcd}\,(a,c)\ \) is an even multiple of \(d\). In view of theorems \ref{1.14} and \ref{3.1} this implies theorem \ref{1.21} as far as it concerns subcase \(1-\).\\
\noindent\emph{Geometric meaning.}
Here the gluing map \(\ \o{A}_N\ \) makes the normal line bundle of the submanifold \(\ f_{L,\o{v}}(M)\ \) (cf. \ref{1.15}) in \(M\) nonorientable. If a Reidemeister class \(\ Q\ \) of \(\ \l(f_{L,\o{v}},f_{0,\o{0}}\r)\ \) (compare \ref{1.19}) is given by an even order orbit then (generic, standard) coincidence components of the corresponding Nielsen class can be removed in a pairwise fashion (see the proof of [Ko5], theorem 1.18). However, in the case of odd order orbits such a deformation of \(\ \l(f_{L,\o{v}},f_{0,\o{0}}\r)\ \) is not possible: \(\ f_{0,\o{0}} \) must necessarily cross \(\ f_{L,\o{v}}(M)\ \) and contribute coincidence data which make \(\ Q\ \) essential (cf. \ref{1.6}).\\

\noindent\textit{Case 2: \(\ a=b=0\ \)}(i.e. \(\ L=0\)).\\
\noindent Here the rank of \(\ G\ =\ \Z^2\ \) exceeds the rank of \(\ \pi_1(S^1)\ \cong\ \Z\ \); thus \(\ \mathrm{MCC}_B(f_1,f_2)=0\ \) by theorem \ref{1.14}. This agrees with our claim since \(\ d=0\ \) or else \(\ |a+b|=0\ \) can only be an even multiple of \(\ d\).\\
\noindent\emph{Geometric meaning.}
By transversality the zero section of \(N\), and hence \(\ f_{0,\o{0}}\ \), can be deformed until they avoid the \(2\)-codimensional curve \(\ f_{L,\o{v}}(M)\ \) in \(\ N\) (compare remark \ref{1.19}).\(\hfill\Box\)\\

\noindent\emph{Proof of theorem \ref{1.22}.} The gluing matrix of \(\ M\ \times_B\ N'\ \) corresponds to \(\ \o{A}_m\oplus\o{A}'\ \) where \(\ \o{A}'= \pm\,\mathrm{id}\ \). If \(\ \varrho_B(f_1,f_2)\ =\ [(G,\beta)]\ \), then
\[
 \varrho_B\l(\hat{f}_1,\hat{f}_2\r)\ =\ \l[\l(G\oplus\Z,\beta\oplus\o{A}'\r)\r].
\]
According to theorem \ref{1.14} the minimum number \(\ \mathrm{MCC}_B\l(\hat{f}_1,\hat{f}_2\r)\ \) can be nontrivial only if \(G\) is finite, i.e. \(\ |a|\neq|b|\). If \(\ N'=K\ \), i.e. \(\ \o{A}'\ =\ -\mathrm{id}\ \), the situation resembles Subcase \(1-\) in the previous proof: there exist infinitely many even order orbits (each lying in \(\ G\oplus\{\pm\, t\}\ \) for some \(\ t\in \Z\ \)), but the odd order orbits are contained in \(\ G\oplus\{0\}\ \) (and can be counted by the methods of Case 0).

If \(\ N'=T\ \), i.e. \(\ \o{A}'=\mathrm{id}\ \), then each orbit of \(\ \beta\ \) gets reproduced at every \(\ \Z\)-level; the orbit numbers, when reduced \(\ \mathrm{mod}\,\infty\ \) (cf. \ref{1.14}), vanish.\\
\noindent\emph{Geometric meaning.}
Each pair \(\ (f_1,f_2)\ \) can be made coincidence free in the extra space provided by \(\ N\ \times_B\ T\ =\ N\ \times\ S^1\ \): just push \(f_1\) and \(f_2\) into different \(S^1\)-levels. This is not always possible in \(\ N\ \times_B\ K\ \) where no coorientation of \(\ N\ \times\ \{0\}\ \) (and hence no canonical direction into which to push) exists and therefore odd order orbits may survive. \(\hfill\Box\)

\subsection*{Acknowledgment}
This research was supported in part by DAAD (Deutscher Akademischer Austauschdienst).

% ------------------------------------------------------------------------
\end{document}